\documentclass{ijmart}
\usepackage{amsmath, amssymb,epic,graphicx,mathrsfs,enumerate}
\usepackage[all]{xy}

\usepackage{amsthm}
\usepackage{amssymb}
\usepackage{latexsym}
\usepackage{longtable}
\usepackage{epsfig}
\usepackage{amsmath}
\usepackage{hhline}
\usepackage{amscd}
\usepackage{newlfont}

\usepackage{color}

 \DeclareMathOperator{\perm}{Sym}
 \DeclareMathOperator{\soc}{soc}

 \DeclareMathOperator{\frat}{Frat}

\DeclareMathOperator{\GL}{GL} \DeclareMathOperator{\gaml}{\Gamma L}

\DeclareMathOperator{\core}{Core}

\DeclareMathOperator{\End}{End} \DeclareMathOperator{\h}{{H^1}}
\DeclareMathOperator{\ider}{InnDer} 
\DeclareMathOperator{\der}{Der} \DeclareMathOperator{\diag}{Diag}

 \newcommand{\del}{\delta}\newcommand{\deltat}{\delta\cdot\theta}
 \newcommand{\ol}{\overline}

\newcommand{\w}{\widetilde}

\newtheorem{thm}{Theorem}
\newtheorem{cor}[thm]{Corollary}

 \newtheorem{lemma}[thm]{Lemma}
\newtheorem{prop}[thm]{Proposition}

\numberwithin{equation}{section}

\renewcommand{\footnote}{\endnote}

\begin{document}
\bibliographystyle{amsplain}
\title[The Chebotarev invariant]{An upper bound on the Chebotarev invariant of a finite group\thanks{Partially supported by Universit\`a di Padova (Progetto di Ricerca di Ateneo: \lq\lq Invariable generation of groups\rq\rq).}}
\author{Andrea Lucchini}
\address{Universit\`a degli Studi di Padova, Dipartimento di Matematica,\\ Via Trieste 63, 35121 Padova, Italy}
\author{Gareth Tracey}
\address{Mathematics Institute, University of Warwick,\\
	Coventry CV4 7AL, United Kingdom}


\begin{abstract}A subset $\{g_1, \ldots , g_d\}$ of a finite group $G$  invariably generates $G$ if the set $\{g_1^{x_1}, \ldots, g_d^{x_d}\}$ generates $G$ for every choice of $x_i \in G$. The Chebotarev invariant $C(G)$ of $G$ is the expected value of the random variable $n$ that is minimal
	subject to the requirement that $n$ randomly chosen elements of $G$ invariably generate $G$. The first author recently showed that $C(G)\le \beta\sqrt{|G|}$ for some absolute constant $\beta$. In this paper we show that, when $G$ is soluble, then $\beta$ is at most $5/3$. We also show that this is best possible. Furthermore, we show that, in general, for each $\epsilon>0$ there exists a constant $c_{\epsilon}$ such that $C(G)\le (1+\epsilon)\sqrt{|G|}+c_{\epsilon}$. 
\end{abstract}

\maketitle
\section{Introduction}
Following \cite{ig} and \cite{dixon}, we say that a subset $\left\{g_{1},g_{2},\hdots,g_{d}\right\}$ of a group $G$ \emph{invariably generates} $G$ if $\left\{g_{1}^{x_{1}},g_{2}^{x_{2}},\hdots,g_{d}^{x_{d}}\right\}$ generates $G$ for every $d$-tuple $(x_{1},x_{2}\hdots,x_{d})\in G^{d}$. The Chebotarev invariant $C(G)$ of $G$ is the expected value of the random variable $n$ that is minimal subject to the requirement that $n$ randomly chosen elements of $G$ invariably generate $G$.

In \cite{kz}, Kowalski and Zywina conjectured that $C(G)=O(\sqrt{|G|})$ for every finite group $G$. Progress on the conjecture was first made in \cite{ig}, where it was shown that $C(G)=O(\sqrt{|G|}\log{|G|})$ (here, and throughout this paper,``$\log$" means $\log$ to base $2$). The conjecture was confirmed by the first author in \cite{ACheboGen}; more precisely, \cite[Theorem 1]{ACheboGen} states that \emph{there exists an absolute constant $\beta$ such that $C(G)\le \beta\sqrt{|G|}$ whenever $G$ is a finite group}.

In this paper, we use a different approach to the problem. In doing so, we show that one can take $\beta=5/3$ when $G$ is soluble, and that this is best possible. Furthermore, we show that for each $\epsilon>0$, there exists a constant $c_{\epsilon}$ such that $C(G)\le (1+\epsilon)\sqrt{|G|}+c_{\epsilon}$. From \cite[Proposition 4.1]{kz}, one can see that this is also (asymptotically) best possible. 

Our main result is as follows   
\begin{thm}\label{KZConjecture} Let $G$ be a finite group. \begin{enumerate}[(i)]
		\item For any $\epsilon>0$, there exists a constant $c_{\epsilon}$ such that $C(G)\le (1+\epsilon)\sqrt{|G|}+c_{\epsilon}$;
		\item If $G$ is a finite soluble group, then
		$C(G)\leq \frac{5}{3}\sqrt{|G|},$ with equality if and only if $G=C_{2}\times C_{2}$.
	\end{enumerate}
\end{thm}

We also derive an upper bound on $C(G)$, for a finite soluble group $G$, in terms of the set of \emph{crowns} for $G$. Before stating this result, we require the following notation: Let $G$ be a finite soluble group. Given an irreducible $G$-module $V$ which is $G$-isomorphic to a complemented chief factor of $G$, let $\delta_V(G)$ be the number of
complemented factors in a chief series of $G$ which are $G$-isomorphic to $V$. Then set $\theta_V(G)=0$ if $\delta_{V}(G)=1$, and $\theta_{V}(G)=1$ otherwise. Also, let $q_{V}(G):=|\End_{G}(V)|$, let $n_{V}(G):=\dim_{\End_{G}(V)}{V}$, and let $H_{V}(G):=G/C_G(V)$ (we will suppress the $G$ in this notation when the group is clear from the context). Also, let $\sigma:=2.118456563\hdots$ be the constant appearing in \cite[Corollary 2]{Pomerance}.
The afore mentioned upper bound can now be stated as follows.
\begin{thm}\label{CheboSolTheorem} Let $G$ be a finite soluble group, and let $A$ [respectively $B$] be a set of representatives for the irreducible $G$-modules which are $G$-isomorphic to a non-central [resp. central] complemented chief factor of $G$. Then
	$$C(G)\le \sum_{V\in A} \min\left\{(\delta_{V}\cdot\theta_{V}+c_{V})|V|,\left(\left\lceil\frac{\delta_{V}\cdot\theta_{V}}{n_{V}}\right\rceil+\frac{q_{V}^{n_V}}{q_{V}^{n_V}-1}\right)|H_{V}|\right\}+\max_{V\in B}{\del_{V}}+\sigma$$
	where $c_{V}:=q_{V}/(q_{V}-1)\le 2$.\end{thm}

The layout of the paper is as follows. In Section 2 we recall the notion of a \emph{crown} in a finite group. In Section 3 we prove Theorem \ref{CheboSolTheorem} and deduce a number of consequences, while Section 4 is reserved for the proof of Theorem \ref{KZConjecture} Part (i). Finally, we prove Theorem \ref{KZConjecture} Part (ii) in Section 5.

\section{Crowns in finite groups}
In Section 2, we recall the notion and the main properties of crowns in finite groups. Let $L$ be   a monolithic primitive group
and let $A$ be its unique minimal normal subgroup. For each positive integer $k$,
let $L^k$ be the $k$-fold direct
product of $L$.
The crown-based power of $L$ of size  $k$ is the subgroup $L_k$ of $L^k$ defined by
$$L_k=\{(l_1, \ldots , l_k) \in L^k  \mid l_1 \equiv \cdots \equiv l_k \ {\mbox{mod}} A \}.$$
Equivalently, $L_k=A^k \diag L^k$. 

Following  \cite{paz},  we say that
two irreducible $G$-groups $V_1$ and $V_2$  are  {$G$-equivalent} and we put $V_1 \sim_G V_2$, if there are
isomorphisms $\phi: V_1\rightarrow V_2$ and $\Phi: V_1\rtimes G \rightarrow V_2\rtimes G$ such that the following diagram commutes:

\begin{equation*}
\begin{CD}
1@>>>V_{1}@>>>V_{1}\rtimes G@>>>G@>>>1\\
@. @VV{\phi}V @VV{\Phi}V @|\\
1@>>>V_{2}@>>>V_{2}\rtimes G@>>>G@>>>1.
\end{CD}
\end{equation*}

\

Note that two $G$\nobreakdash-isomorphic
$G$\nobreakdash-groups are $G$\nobreakdash-equivalent. In the particular case where $V_1$ and $V_2$ are abelian the converse is true:
if $V_1$ and $V_2$ are abelian and $G$\nobreakdash-equivalent, then $V_1$
and $V_2$ are also $G$\nobreakdash-isomorphic.
It is proved (see for example \cite[Proposition 1.4]{paz}) that two  chief factors $V_1$ and $V_2$ of $G$ are  $G$-equivalent if and only if  either they are  $G$-isomorphic between them or there exists a maximal subgroup $M$ of $G$ such that $G/\core_G(M)$ has two minimal normal subgroups $N_1$ and $N_2$
$G$-isomorphic to $V_1$ and $V_2$ respectively. For example, the  minimal normal subgroups of a crown-based power $L_k$ are all $L_k$-equivalent.

Let $V=X/Y$ be a chief factor of $G$. A complement $U$ to $V$ in $G$ is a subgroup $U $ of $G$ such that $UV=G$ and $U \cap X=Y$. We say that   $V=X/Y$ is a Frattini chief factor if  $X/Y$ is contained in the Frattini subgroup of $G/Y$; this is equivalent to saying that $V$ is abelian and there is no complement to $V$ in $G$.
The  number $\delta_V(G)$  of non-Frattini chief factors $G$-equivalent to $V$   in any chief series of $G$  does not depend on the series. Now,
we denote by  $L_V$  the  monolithic primitive group  associated to $V$,
that is
$$L_{V}=
\begin{cases}
V\rtimes (G/C_G(V)) & \text{ if $V$ is abelian}, \\
G/C_G(V)& \text{ otherwise}.
\end{cases}
$$
If $V$ is a non-Frattini chief factor of $G$, then $L_V$ is a homomorphic image of $G$.
More precisely,  there exists
a normal subgroup $N$ of $G$ such that $G/N \cong L_V$ and $\soc(G/N)\sim_G V$. Consider now  all the normal subgroups $N$ of $G$ with the property that  $G/N \cong L_V$ and $\soc(G/N)\sim_G V$:
the intersection $R_G(V)$ of all these subgroups has the property that  $G/R_G(V)$ is isomorphic to the crown-based  power $(L_V)_{\delta_V(G)}$.
The socle $I_G(V)/R_G(V)$ of $G/R_G(V)$ is called the $V$-crown of $G$ and it is  a direct product of $\delta_V(G)$ minimal normal subgroups $G$-equivalent to $V$. 

\begin{lemma}{\cite[Lemma 1.3.6]{classes}}\label{corona}
	Let $G$ be a finite  group with trivial Frattini subgroup. There exists
	a chief factor $V$ of $G$ and a non trivial normal subgroup $U$ of $G$ such that $I_G(V)=R_G(V)\times U.$
\end{lemma}

\begin{lemma}{\cite[Proposition 11]{crowns}}\label{sotto} Assume that $G$ is a finite  group with trivial Frattini subgroup and let $I_G(V), R_G(V), U$ be as in the statement of Lemma \ref{corona}. If $KU=KR_G(V)=G,$ then $K=G.$
\end{lemma}

\section{Crown-based powers with abelian socle}
The aim of this section is to prove Theorem \ref{CheboSolTheorem}. For a finite group $G$ and an irreducible $G$-group $V$, we write $\Omega_{G,V}$ for the set of maximal subgroups $M$ of $G$ such that either $\soc{(G/\core_{G}(M))}\sim_{G}V$ or $\soc{(G/\core_{G}(M))}\sim_G V\times V$. Also, for $M\in \Omega_{G,V}$, we write $\w{M}$ for the union of the $G$-conjugates of $M$. We will also say that the elements $g_{1}$, $g_{2}$, $\hdots$, $g_{k}\in G$ \emph{satisfy the $V$-property in $G$} if $g_{1}$, $g_{2}$, $\hdots$, $g_{k}\in \w{M}$ for some $M\in \Omega_{V}$. Finally, let $P_{G,V}^{\ast}(k)$ denote the probability that $k$ randomly chosen elements of $G$ satisfy the $V$-property in $G$.

Suppose now that $V$ is abelian, and consider the faithful irreducible linear group $H:=G/C_G(V)$. We will denote by $\der(H,V)$ the set of the derivations from
$H$ to $V$ (i.e. the maps $\zeta: H\to V$ with the property that $\zeta(h_1h_2)=\zeta(h_1)^{h_2}+\zeta(h_2)$ for every $h_1,h_2\in H$). If $v\in V$ then the map $\zeta_v:H\to V$
defined by $\zeta_v(h)=[h,v]$ is a derivation, called an \emph{inner derivation} from $H$ to $V$. The set  $\ider(H,V)=\{\zeta_v\mid v \in V\}$ of the inner derivations from $H$ to $V$ is a subgroup of $\der(V,H)$ and the factor group $\h(H,V)=\der(H,V)/\ider(H,V)$ is the first cohomology group of $H$
with coefficients in $V.$ 

\begin{prop}\label{crucial} Let $H$ be a group acting faithfully and irreducibly on an elementary abelian $p$-group $V$. For a positive integer $u$, we consider the semidirect product $G=V^{u}\rtimes H$ where the action of $H$ is diagonal on $V^{u}$; that is, $H$ acts in the same away on each of the $u$ direct factors. Assume also that $u=\delta_{V}(G)$. View $V$ as a vector space over the field $F=End_{H}(V)$. Let $h_{1},\hdots,h_{k}\in H$, and $w_{1},\hdots,w_{k}\in V^{u}$, and write $w_{i}=(w_{i,1},w_{i,2},\hdots,w_{i,u})$. Assume that $h_{1}w_{1},h_{2}w_{2},\hdots,h_{k}w_{k}$ satisfy the $V$-property in $G$. Then for $1\le j\le u$, the vectors
	$$r_{j}:=(w_{1,j},w_{2,j},\hdots,w_{k,j})$$
	of $V^{k}$ are linearly dependent modulo the subspace $W+D$, where
	$$\begin{aligned}W &:=\left\{(y_{1},y_{2},\hdots,y_{k})\text{ : }y_{i}\in [h_{i},V]\text{ for }1\le i\le k\right\} \text{, and}\\
	D &:= \left\{(\zeta(h_1),\zeta(h_2),\hdots,\zeta(h_k) )\in V^{k}\text{ : }\zeta\in\der(H,V)\right\}.\end{aligned}$$
\end{prop}
\begin{proof} Let $M$ be a maximal subgroup of $G$ such that $M\in \Omega_{V}$, and $h_{1}w_{1},\hdots,$ $h_{k}w_{k}\in \widetilde{M}$. Since $u=\delta_{V}(G)$, $M$ cannot contain $V^{u}$, and hence $MV^{u}=G$. Thus, $M/M\cap V^{u}\cong H$, and hence there exists an integer $t\ge 0$ and elements $h_{k+1}w_{k+1},\hdots,h_{k+t}w_{k+t}\in M$ such that $h_{1},\hdots,h_{k},h_{k+1},\hdots,h_{k+t}$ invariably generate $H$. But then, \cite[Proposition 6]{ACheboGen} implies, in particular, that $r_{1},\hdots,r_{u}\in V^{k}$ are linearly dependent modulo $W+D$, as needed.\end{proof}

Before proceeding to the proof of Theorem \ref{CheboSolTheorem}, we require the following easy result from probability theory.
\begin{prop}\label{binomale} Write $B(k,p)$ for the binomial random variable with $k$ trials and probability $0<p\le 1$. Fix $l\ge 0$. Then
	$$\sum_{k=l}^{\infty} P(B(k,p)=l)\le \frac{1}{p}.$$\end{prop} 
\begin{proof} Note first that
	$$\binom{k}{l}x^{k-l}=\frac{1}{l!}\frac{d^{l}}{dx^{l}}x^{k}$$
	where $\frac{d^{l}}{dx^{l}}x^{k}$ denotes the $l$-th derivative of $x^{k}$. Let $x=1-p$. By definition, $P(B(k,p)=l)=\binom{k}{l}(1-x)^{l}x^{k-l}$. Thus\begin{align*}
	\sum_{k=l}^{\infty}P(B(k,p)=l) &=(1-x)^{l}\sum_{k=l}^{\infty}\binom{k}{l}x^{k-l}\\
	&=\frac{(1-x)^{l}}{l!}\sum_{k=l}^{\infty}\frac{d^{l}}{dx^{l}}x^{k}\\
	&=\frac{(1-x)^{l}}{l!}\frac{d^{l}}{dx^{l}}\sum_{k=l}^{\infty}x^{k}\\
	&\le\frac{(1-x)^{l}}{l!}\frac{d^{l}}{dx^{l}}\frac{1}{1-x}\\
	&=\frac{(1-x)^{l}}{l!}\frac{l!}{(1-x)^{(l+1)}}=\frac{1}{1-x}=\frac{1}{p}\end{align*}
	as needed. (Note that the third equality above follows since the series $\sum_{k=l}^{\infty}x^{k}$ is convergent.)\end{proof}

We shall also require the following. We remark that since $P^*_{G,V}(k)\le  \sum_{\widetilde{M}\in \Omega_{V}}\left(\frac{|\widetilde{M}|}{|G|}\right)^{k}$ and $\frac{|\widetilde{M}|}{|G|}<1$, $\sum_{k=0}^{\infty} P^*_{G,V}(k)$ converges.
\begin{prop}\label{firstred} Let $G$ be a finite group, and let $A$ [respectively $B$] be a set of representatives for the irreducible $G$-groups which are $G$-equivalent to a non-central [resp. central] non-Frattini chief factor of $G$. Then\begin{enumerate}
		\item $C(G)\le \sum_{V\in A}\sum_{k=0}^{\infty} P^*_{G,V}(k)+\max_{V\in B}\delta_{V}+\sigma$, and;
		\item If $\frat(G)=1$ and $U$ and $V$ are as in Lemma \ref{corona}, then $C(G)\le C(G/U)+\sum_{k=0}^{\infty} P^*_{G,V}(k)$.\end{enumerate} 
\end{prop}
\begin{proof} By definition, $C(G)=\sum_{k=0}^{\infty} (1-P_{I}(G,k))$, where $P_{I}(G,k)$ denotes the probability that $k$ randomly chosen elements of $G$ invariably generate $G$. Let $P_{G,G/G'}(k)$ denote the probability that $k$ randomly chosen elements $g_1$, $\hdots$, $g_k$ of $G$ satisfy $\langle G'g_1,\hdots, G'g_k\rangle=G$. Then it is easy to see that 
	\begin{equation}\label{firstredeq}1-P_{I}(G,k)\le 1-P_{G,G/G'}(k)+\sum_{V\in A} P^{\ast}_{G,V}(k).\end{equation}
	Clearly $P_{G,G/G'}(k)$ is the probability that a random $k$-tuple of elements from $G/G'$ generates $G/G'$. Hence, $C(G/G')=\sum_{k=0}^{\infty} (1-P_{G,G/G'}(k))$ is at most $d(G/G')+\sigma$ by \cite[Corollary 2]{Pomerance} (here, for a group $X$, $d(X)$ denotes the minimal number of elements required to generate $X$). Since $d(G/G')\le \max_{V\in B}\delta_V$, it follows from (\ref{firstredeq}) that $C(G)\le\max_{V\in B}\delta_V+\sigma+\sum_{V\in A}\sum_{k=0}^{\infty} P^{\ast}_{G,V}(k)$, and Part (i) follows.
	
	Assume that $\frat(G)=1$, and let $U$ and $V$ be as in Lemma \ref{corona}. Then 
	\begin{equation}\label{nomoremax0}1-P_{I}(G,k)\le 1-P_{I}(G/U,k)+\sum_{W}P^*_{G,W}(k)\end{equation}
	where the sum in the second term goes over all complemented chief factors $W$ of $G$ not containing $U$. Now, if $M$ is a maximal subgroup of $G$ not containing $U$, then $M$ contains $R_{G}(V)$, by Lemma \ref{sotto}. Hence, $\core_G(M)$ contains $R_G(V)$, so $M\in \Omega_{G,V}$. Since $C(G)=\sum_{k=0}^{\infty} (1-P_{I}(G,k))$, Part (ii) now follows immediately from (\ref{nomoremax0}), and this completes the proof.
\end{proof}

The proof of Theorem \ref{CheboSolTheorem} will follow as a corollary of the proof of the next proposition. For a finite group $G$, and an abelian chief factor $V$ of $G$, set $H_{V}=H_{V}(G):=G/C_G(V)$, $m=m_V=m_{V}(G):=\dim_{\End_{G}(V)}\h(H_{V},V)$, and write $p=p_{V}=p_{V}(G)$ for the probability that a randomly chosen element $h$ of $H_{V}$ fixes a non zero vector in $V$. Also, let $\delta_{V}=\delta_V(G)$ be the number of
complemented factors in a chief series of $G$ which are $G$-isomorphic to $V$, and set $\theta_{V}=\theta_V(G)=0$ if $\delta_{V}=1$, and $\theta_{V}=1$ otherwise. Finally, let $q_{V}=q_{V}(G):=|\End_{G}(V)|$ and $n_{V}=n_{V}(G):=\dim_{\End_{G}(V)}{V}$.
\begin{prop}\label{CheboProp} Let $G$ be a finite group with trivial Frattini subgroup, and let $U$, $V$ and $R=R_G(V)$ be as in Lemma \ref{corona}. If $V$ is nonabelian, then set $\alpha_U:=\sum_{k=0}^{\infty} P^*_{G,V}(k)$. If $V$ is abelian, then write $q=q_V$, $n=n_V$ and $H=H_V$, $p=p_V$ and $m=m_V$. Also, set $\delta=\delta_{V}$ and define $\theta=0$ if $\delta=1,$ $\theta=1$ otherwise, and set
	
	$$\alpha_U:=\begin{cases}\sum_{0\leq i\leq \delta-1}\frac{q^\delta}{q^\delta-q^i}\leq \delta+\frac{q}{(q-1)^2}& \text { if }H=1,\\
	\min\left\{\left(\delta\cdot \theta+m+\frac{q}{q-1}\right)\frac{1}{p},\left(\lceil\frac{\delta\cdot \theta}{n}\rceil+\frac{q^n}{q^n-1}\right)|H|\right\}& \text { otherwise.}
	\end{cases}$$
	Then
	$$C(G)\le C(G/U)+\alpha_U.$$
\end{prop}
\begin{proof} By Proposition \ref{firstred} Part (ii), we have
	\begin{equation}\label{nomoremax} C(G)\le C(G/U)+\sum_{k=0}^{\infty} P^*_{G,V}(k).\end{equation}
	
	Thus, we just need to prove that $\sum_{k=0}^{\infty} P^{\ast}_{G,V}(k)\le \alpha_U$. Therefore, we may assume that $V$ is abelian. Writing bars to denote reduction modulo $R_G(V)$, note that if $M$ is a maximal subgroup of $G$ with $M\in \Omega_{G,V}$, then $R_G(V)\le M$ and $\overline{M}\in \Omega_{\overline{G},V}$. Hence, $P^{\ast}_{G,V}(k)\le P^{\ast}_{\ol{G},V}(k)$, so we may assume that $R_G(V)=1$. Thus, $G\cong V^{\delta}\rtimes H$, where $H$ acts faithfully and irreducibly on $V$, and diagonally on $V^{\delta}$.
	
	Suppose first that $|H|=1$. Then $G=V^{\delta}\cong (C_r)^{\delta}$, for some prime $r$, and $P^*_{G,V}(k)$ is the probability that $k$ randomly chosen elements of $G$ fail to generate $G$. Hence, $\sum_{k=0}^{\infty} P^*_{G,V}(k)$ is the expected number of random elements to generate $(C_r)^{\delta}$, which is well known to be
	$$\sum_{i=0}^{\delta-1}\frac{r^{\delta}}{r^{\delta}-r^{i}}.$$
	See, for instance, \cite[top of page 193]{Pomerance}.
	
	So we may assume that $|H|>1$. Let $F=\End_HV$, so that $|F|=q$, $\dim_FV=n$, and $|V|=q^n$. Fix elements $x_{1}$, $x_{2}$ ,$\hdots$, $x_{k}$ in $G$, and for $i\in\{1,\dots,k\},$ let $x_i=w_ih_i$ with $w_i\in V^\delta$ and $h_i\in H.$ For $t\in \{1,\dots,\delta\}$ let
	$$\begin{aligned}r_t&=(\pi_t(w_1),\dots,\pi_t(w_k))\in V^{k}.\end{aligned}$$
	where $\pi_{t}$ denotes projection onto the $t$-th direct factor of $V^{\delta}$.
	Moreover let
	\begin{align*}W &:=\left\{(u_{1},u_{2},\hdots,u_{k})\text{ : }u_{i}\in [h_{i},V]\text{ for }1\le i\le k\right\}\text{, and}\\
	D &:= \left\{(\zeta(h_1),\zeta(h_2),\hdots,\zeta(h_k) )\in V^{k}\text{ : }\zeta\in\der(H,V)\right\}.\end{align*}
	By Proposition \ref{crucial}, $P^{\ast}_{G,V}(k)$ is at most the probability that
	$r_1,\dots,r_\delta$ are linearly dependent modulo $W+D$. Also, for an $f$-tuple $J:=(j_{1},j_{2},\hdots,j_{f})$ of distinct elements $j_i$ of $\left\{1,\hdots,k\right\}$, set 
	$$r_{t,J}:=(\pi_t(w_{j_1}),\pi_t(w_{j_2}),\hdots,\pi_t(w_{j_f}))\in V^{f}$$
	for $t\in \left\{1,\hdots,\delta\right\}$, and set 
	\begin{align*}W_{J} &:=\left\{(u_{j_1},u_{j_2},\hdots,u_{j_f})\in V^{f}\text{ : }u_{i}\in [h_{j_i},V]\text{ for }1\le i\le f\right\}\text{, and}\\
	D_{J} &:= \left\{(\zeta(h_{j_1}),\zeta(h_{j_2}),\hdots,\zeta(h_{j_f}) )\in V^{f}\text{ : }\zeta\in\der(H,V)\right\}.\end{align*}
	Notice that: $(\ast)$ If $J$ is fixed and $r_1,\dots,r_\delta$ are $F$-linearly dependent modulo $W+D$, then the vectors $r_{1,J},\dots,r_{\delta,J}$ of $V^{f}$ are $F$-linearly dependent modulo $W_J+D_J$.
	
	We will prove first that 
	\begin{align}\sum_{k=0}^{\infty} P^{\ast}_{G,V}(k)\le (\deltat +m+c_{V})\frac{1}{p},\end{align}
	where $c_V$ is as in the statement of Theorem 2. To this end, let $\Delta_{l}$ be the subset of $H^k$ consisting of the $k$-tuples $(h_1,\dots,h_k)$ with the property that $C_V(h_i)\neq 0$ for precisely $l$ different choices of $i\in\{1,\dots,k\}.$  If $(h_1,\dots,h_k)\in \Delta_{l},$ then, by \cite[Lemma 7]{ACheboGen}, $W+D$ is a subspace of $V^{k}\cong F^{nk}$ of codimension at least $l-m$: so the probability that $r_{1},\dots,r_{\delta}$ are $F$-linearly dependent modulo $W+D$ is at most
	$$\begin{aligned}p_{l}&=1-\left(\frac{q^{nk}-q^{nk-l+m}}{q^{nk}}\right)
	\cdots\left(\frac{q^{nk}-q^{nk-l+m+\delta-1}}{q^{nk}}\right)\\
	&=1-\left(1-\frac{1}{q^{l-m}}\right)\dots \left(1-\frac{q^{\delta-1}}{q^{l-m}}\right)\\&\leq \min\left\{1,\left(\frac{q^\delta-1}{q-1}\right)\frac{1}{q^{l-m}}\right\}\le \min\left\{1,1/q^{l-m-\deltat}\right\}.
	\end{aligned}$$
	Hence, we have 
	$$\begin{aligned}\sum_{k=0}^{\infty} P^{\ast}_{G,V}(k) &\le \sum_{k=0}^{\infty}\sum_{l=0}^{k} P(B(k,p)=l)\min\left\{1,q^{\deltat+m-l}\right\}\\
&\le \sum_{k=0}^{\infty} P(B(k,p)<\deltat+m )+\sum_{k=0}^{\infty}\sum_{l=\deltat+m}^{k} P(B(k,p)=l)q^{\deltat+m-l}\\
	&\le \sum_{k=0}^{\infty} P(B(k,p)<\deltat+m )+\sum_{l=0}^{\infty}q^{-l}\!\!\!\!\!\sum_{k=l+\deltat+m}^{\infty} \!\!\!\!\!\!P(B(k,p)=l+\deltat+m)\\
	&\le \frac{\deltat+m+c_{V}}{p} \end{aligned}$$
	where $c_{V}=\frac{q}{q-1}$. Note that the last step above follows from Proposition \ref{binomale}. 
	
	Thus, all that remains is to show that
	\begin{align}\sum_{k=0}^{\infty} P^{\ast}_{G,V}(k)\le \left(\left\lceil\frac{\deltat}{n}\right\rceil+\frac{q^n}{q^{n}-1}\right)|H|.\end{align}
	For this, we define $\Omega_{l}$ to be the subset of $H^k$ consisting of the $k$-tuples $(h_1,\dots,h_k)$ with the property that $h_i=1$ for precisely $l$ different choices of $i\in \{1,\dots,k\}.$ Suppose that $(h_1,\hdots,h_k) \in \Omega_l$, and set $J:=(j_{1},j_{2},\hdots,j_{l})$, where $j_{1}<j_{2}<\hdots<j_{l}$ and $\left\{j_{1},j_{2},\hdots,j_{l}\right\}=\left\{i\text{ }|\text{ }1\le i\le k, h_i=1\right\}$. Then, by $(\ast)$, the probability $p_{l}'$ that $r_{1}$, $r_{2}$, $\hdots$, $r_{\delta}$ are $F$-linearly dependent modulo $W+D$ is at most the probability that the vectors $r_{1,J}$, $r_{2,J}$, $\hdots$, $r_{\delta,J}\in V^{l}$ are $F$-linearly dependent modulo $W_{J}+D_{J}$. But $W_{J}+D_{J}=0$, by the definition of $J$. Thus we have
	$$\begin{aligned}p_{l}'&\le 1-\left(\frac{q^{nl}-1}{q^{nl}}\right)
	\cdots\left(\frac{q^{nl}-q^{nl-\delta-1}}{q^{nl}}\right)\\
	&=1-\left(1-\frac{1}{q^{nl}}\right)\dots \left(1-\frac{q^{\delta-1}}{q^{nl}}\right)\leq \min\left\{1,\left(\frac{q^\delta-1}{q-1}\right)\frac{1}{q^{nl}}\right\}\le \min\left\{1,\frac{1}{
		q^{nl-\deltat}}\right\}.
	\end{aligned}$$
	Hence, if $\alpha:=\lceil \frac{\deltat}{n}\rceil$, and $p'=1/|H|$ is the probability that a randomly chosen element of $H$ is the identity, then we have
	\begin{align*} \sum_{k=0}^{\infty} P^{\ast}_{G,V}(k) &\le \sum_{k=0}^{\infty} P(B(k,p')<\alpha )+\sum_{k=0}^{\infty}\sum_{l=\alpha}^{k} P(B(k,p')=l)q^{\deltat-nl}\\
	&\le \sum_{k=0}^{\infty} P(B(k,p')<\alpha )+\sum_{l=0}^{\infty}q^{-nl-n\alpha+\deltat}\sum_{k=l+\alpha}^{\infty} P(B(k,p')=l+\alpha)\\
	&\le \sum_{k=0}^{\infty} P(B(k,p')<\alpha )+\sum_{l=0}^{\infty}q^{-nl}\sum_{k=l+\alpha}^{\infty} P(B(k,p')=l+\alpha)\\
	&\le \frac{1}{p'}\left(\alpha +\frac{q^{n}}{q^{n}-1}\right) \end{align*}
	Note that the last step above again follows from Proposition \ref{binomale}. Since $p'=1/|H|$, (3.5) follows, whence the result.
\end{proof}

We are now ready to prove Theorem \ref{CheboSolTheorem}.
\begin{proof}[Proof of Theorem \ref{CheboSolTheorem}] By Proposition \ref{firstred} Part (i), we have
	$$C(G)\le\max_{V\in B}\delta_V+\sigma+\sum_{k=0}^{\infty}\sum_{V\in A} P^{\ast}_{G,V}(k)$$
	Thus, it will suffice to prove that 
	\begin{align} \sum_{k=0}^{\infty} P^{\ast}_{G,V}(k)\le \min\left\{(\delta_{V}+c_{V})q_{V}^{n_{V}},\left(\left\lceil\frac{\delta_{V}}{n_V}\right\rceil+\frac{q_{V}^{n_{v}}}{q_{V}^{n_{V}}-1}\right)|H_{V}|\right\}\end{align} for each non-central complemented chief factor $V$ of $G$.
	
	However, since $\h(H,V)=0$ by \cite[Lemma 1]{st}, and since $p_{V}\le |H_{v}|/|H|\le 1/|V|$ (for any non-zero vector $v\in V$), this follows immediately from the proof of Proposition \ref{CheboProp}.\end{proof}

\begin{cor}\label{DelAstCor} Let $G$ be a finite soluble group, and let $A$ and $B$ be as in Theorem \ref{CheboSolTheorem}. Then 
	$$C(G)\le d(G)\sum_{V\in A}\left(1+\frac{{q_V^{n_{V}}}|H_V|}{{q_V^{n_{V}}}-1}\right)+\sigma.$$\end{cor} 
\begin{proof} For $V\in A\cup B$, set $\gamma_{V}:=\lceil\delta_{V}/n_{V}\rceil$, and $p'_{V}=1/|H_{V}|$. Note also that $n_V=|H_V|=1$ when $V\in B$. Arguing as in the last paragraph of the proof of Proposition \ref{CheboProp}, we have 
	\begin{align*}C(G) &\le \sum_{V\in A}\sum_{k=0}^{\infty}\sum_{l=0}^{k} \min\{q_V^{-n_{V}l+\delta_V},1 \} P(B(k,p'_V)=l)+ \max_{V\in B}{\del_{V}}\!+\sigma\\
	&\le\sum_{V\in A}\sum_{k=0}^{\infty}\!P(B(k,p'_V)\!<\!\gamma_{V})\!+\!\!\sum_{V\in A}\sum_{l=0}^{\infty}q_V^{-n_{V}l}\!\!\!\!\!\sum_{k=l+\gamma_V}^{\infty}\!\!\!\!\!P(B(k,p'_V)=l+\gamma_{V} )+\\
&\quad \max_{V\in B}{\del_{V}}\!+\sigma\\
	&\le \sum_{V\in A} \gamma_{V}/p'_{V} +\sum_{V\in A}\frac{{q_{V}}^{n_{V}}}{{p'_{V}(q_{V}}^{n_{V}}-1)}+\max_{V\in B}{\del_{V}}+\sigma\\
	&\le \left(\max_{V\in A\cup B}\gamma_{V}\right)\sum_{V\in A}\left(1+\frac{{q_{V}}^{n_{V}}}{{q_{V}}^{n_{V}-1}}\right) |H_{V}|+\sigma. \end{align*}
	We remark that the third inequality above follows from Proposition \ref{binomale}. Finally, \cite[Theorem 1.4 and paragraph after the proof of Theorem 2.7]{ADV} imply that $d(G)=\max_{V\in A\cup B}\left\{1+a_{V}+\left\lfloor\frac{{\delta}_V-1}{n_V}\right\rfloor\right\}$, where $a_{V}=0$ if $V\in B$, and $a_{V}=1$ otherwise. In particular, $d(G)\ge \max_{V\in A\cup B}{\gamma_{V}}$, and the result follows. \end{proof}

\section{Proof of Theorem \ref{KZConjecture} Part (i)} 
Before proceeding to the proof of Part (i) of Theorem \ref{KZConjecture}, we require the following result, which follows immediately from the arguments used in \cite[Proof of Proposition 10]{ACheboGen}.
\begin{prop}\label{pH}{\cite[Proof of Proposition 10]{ACheboGen}} Let $H$ be a finite group acting faithfully and irreducibly on an elementary abelian group $V$, and denote by $p$ the probability that a randomly chosen element $h$ of $H$ centralises a non zero vector of $V$. Also, write $m:=\dim_{\End_{H}(V)}\h(H,V)$. Assume that $\h(H,V)$ is nontrivial and that $|H|\geq |V|$. Then there exists an absolute constant $C$ such that $p|H|\ge 2(m+1)^{2}$ if $|H|\ge C$.\end{prop}

\begin{proof}[Proof of Theorem \ref{KZConjecture} Part (i)] Since $C(G)=C(G/\frat(G))$, we may assume that $\frat(G)=1$. Thus, Proposition \ref{CheboProp} applies: adopting the same notation as used therein, we have 
	\begin{equation}\label{main}
	C(G)\leq C(G/U)+\alpha_U.
	\end{equation} 
	
	Using (\ref{main}), the proof of the theorem reduces to proving that
	\begin{equation}\label{red}\alpha_U\leq (1+\beta_U)\sqrt{|{G}|}\end{equation}
	where $\beta_{U}\to 0$ as $|U|\to \infty.$ Indeed, suppose that (\ref{red}) holds, fix $\epsilon>0$, and suppose that Theorem \ref{KZConjecture} holds for groups of order less than $|G|$. Then since $|U|>1$, there exists a constant $c_{\epsilon}$ such that $C(G/U)\le (1+\epsilon)\sqrt{|G/U|}+c_{\epsilon}$. Hence, by (\ref{main}) and (\ref{red}) we have $C(G)\le (1+\beta_{U}+\frac{1+\epsilon}{\sqrt{|U|}})\sqrt{|G|}+c_{\epsilon}$. It is now clear that by choosing $|U|$ to be large enough, we have $C(G)\le (1+\epsilon)\sqrt{|G|}+c_{\epsilon}$, as needed. 
	
	Assume first that $U$ is nonabelian. By \cite[Proof of Lemma 13]{ACheboGen}, there exist absolute constants $c_{1}$ and $c_{2}$ such that 
	$$P^*_{G,V}(k) \le \min\left\{1,c_{1}\sqrt{|{G}|^{3}}(1-c_{2}/\log{|{{G}}|})^{k}\right\}.$$
	Also, there exists a constant $c_{3}$ such that if $k\ge c_{3}(\log{|{G}|})^{2}$, then\\
	$c_{1}\sqrt{|{G}|^{3}}(1-c_{2}/\log{|{G}|})^{k}$ tends to $0$ as ${|G|}$ tends to $\infty$. It follows that
	\begin{align*} \alpha_U &= \sum_{k=0}^{\infty}P^*_{G,V}(k)\\
	&\le \lceil c_{3}(\log{|{G}|})^{2}\rceil +c_{1}\sqrt{|{G}|^{3}}(1-c_{2}/\log{|{G}|})^{\lceil c_{3}(\log{|{G}|})^{2}\rceil}\sum_{k=0}^{\infty}(1-c_{2}/\log{|{G}|})^{k}\\
	&= \lceil c_{3}(\log{|{G}|})^{2}\rceil +\frac{c_{1}}{c_{2}}\sqrt{|{G}|^{3}}\log{|G|}(1-c_{2}/\log{|{G}|})^{\lceil c_{3}(\log{|{G}|})^{2}\rceil}\end{align*}
	and (\ref{red}) holds.
	
	So we may assume that $U$ is abelian, and hence $|{G}|\ge |V|^{\delta}|H|$. The inequality (\ref{red}) then follows easily from the definition of $\alpha_U$, except when $\delta=1$ and $|H|\geq |V|.$ Indeed, if $|H|\le |V|$ and $\delta=1$, then $\frac{q^n}{q^{n}-1}\to 1$ as $|U|=q^{n}\to\infty$; if $|H|\le |V|$ and $\delta>1$, then 
	$$\left(\left\lceil \frac{\delta}{n}\right\rceil+\frac{q^{n}}{q^{n}-1}\right)|H|\le \frac{\left\lceil\frac{ \delta}{n}\right\rceil+\frac{q^n}{q^{n}-1}}{|V|^{\frac{\delta-1}{2}}}\sqrt{|G|}$$ 
	which clearly gives us what we need, since $|U|=|V|^{\delta}$ is tending to $\infty$. The other cases are similar.
	
	So assume that $\delta=1$ and $|H|\geq |V|$. We distinguish two cases:\begin{enumerate}[(1)]
		\item $m\neq 0$ and $|V|\le |H|\le (m+1)^{2}|V|$. Denote by $p$ the probability that
		a randomly chosen element $h$ of $H$ centralizes a non-zero vector of $V$: By Proposition \ref{pH}, there exists an absolute constant $C$ such that $p|H|\ge 2(m+1)^{2}$ if $|H|\ge C$. Thus,
		$$\alpha_U\le \left(m+\frac{q}{q-1}\right)\frac{1}{p}\le (m+2)\frac{|H|}{2(m+1)^{2}}\le \frac{|H|}{m+1}\le \sqrt{|H||V|}$$
		if $|H|\ge C$, from which (\ref{red}) follows.
		
		\item $|H| \ge |V|(m + 1)^2$. We remark first that, for any fixed nonzero vector $v$ in $V$, we have $p\ge \frac{|H_v|}{|H|}$, where $H_{v}$ denotes the stabiliser of $v$ in $H$. If $H$ is not a transitive linear group, then there is an
		orbit $\Omega$ for the action of $H$ on $V\backslash\left\{0\right\}$ with $|\Omega|\le q^{n}/2$. Choose $v\in \Omega$: we have
		$$\frac{1}{p}\le \frac{|H|}{|H_v|}\le \frac{q^n}{2},$$
		hence
		$$\alpha_U\le \frac{m+2}{p}\le (m+1)q^{n}\le \sqrt{|H||V|}.$$
		We remain with the case when $H$ is a transitive linear group. There
		are four infinite families:\begin{enumerate}[(a)]
			\item $H\le \Gamma L(1,q^{n})$;
			\item $SL(a,r)\unlhd H$, where $r^{a}=q^{n}$;
			\item $Sp(2a,r)\unlhd H$, where $a\ge 2$ and $r^{2a}=q^{n}$;
			\item $G_2(r)\unlhd H$, where $q$ is even, and $q^{n}=r^{6}$.\end{enumerate}
		Furthermore, $H$ and $m$ are exhibited in \cite[Table 7.3]{Cameron}: in each case, we have $m\le 1$. Furthermore, we have $|H|=(q^{n}-1)\rho$, where $\rho$ is the order of a point stabiliser. Hence, if $\rho\ge 9$, then
		$$\alpha_{U}\le \frac{m+2}{p}\le \frac{3}{p}\le 3|V|\le \sqrt{|H||V|}.$$
		So we may assume that $\rho\le 8$. Suppose first that (a) holds. Then $H$ is soluble, so $m=0$. Also, $\rho=|H_{v}|\le 8$ implies that $n\le 8$. Hence, as $q^{n}\le q^{8}$ approaches $\infty$, $\frac{q}{q-1}$ approaches $1$, and (\ref{red}) follows since
		$$\alpha_U\le \frac{q}{q-1}|V|\le \frac{q}{q-1}\sqrt{|H||V|}.$$
		So we may assume that (a) does not hold. In particular, if (b) or (c) holds then $a\ge 2$. It follows (in either of the cases (b), (c) or (d)) that if $q^n$ is large enough, then $|H|\ge 9q^{n}$,
		and so
		$$\alpha_{U}\le \frac{(m+2)}{p}\le 3q^{n}\le \sqrt{|H||V|}.$$
		This gives us what we need, and completes the proof.\qedhere\end{enumerate}\end{proof}

\section{Proof of Theorem \ref{KZConjecture} Part (ii)}
In this section, we prove Part (ii) of Theorem \ref{KZConjecture} in a number of steps. The first is as follows:
\begin{lemma}\label{tecn}Let $G$ be a finite soluble group with trivial Frattini subgroup, and let $U$ and $V$ be as in Lemma \ref{corona}. Assume that $V$ is abelian and non-central in $G$, and let $H=H_V$. Then
	$$\frac{\alpha_U}{|G|^{1/2}}< \frac{5}{3}\left(\frac{|U|^{1/2}-1}{|U|^{1/2}}\right)$$
	except  when $|H|<|V|$ and one of the following cases occur:
	\begin{enumerate}
		\item $\delta=2$, $q^n=4$ and $|R_G(V)|=1.$
		\item $\delta=2$, $q^n=3$ and $|R_G(V)|\leq 2.$
		\item $\delta=1$, $4\leq q^n \leq 7$ and $|R_G(V)|=1.$
		\item $\delta=1$, $q^n=3$ and $|R_G(V)|\leq 3.$
	\end{enumerate}
\end{lemma}
\begin{proof} Note that $m=0$ since $H$ is soluble. We distinguish the following cases:
	
	\noindent
	Case 1) $|H| < |V|$ and $\delta\neq 1.$
	Since, $|G|=\lambda|H||V|^\delta$ for some positive integer $\lambda$ it suffices to prove
	\begin{equation}\frac{3\left(\delta+\frac{q^n}{q^n-1}\right)
		\left(\frac{q^{n\delta/2}}{q^{n\delta/2}-1}\right)|H|}{5\lambda^{1/2}|H|^{1/2}q^{n\delta/2}}\leq \frac{3}{5\lambda^{1/2}}\left(\delta+\frac{q^n}{q^n-1}\right)
	\left(\frac{(q^{n}-1)^{1/2}}{q^{n\delta/2}-1}\right)< 1.
	\end{equation}
	If $\delta\geq 3$ then
	$$\frac{3}{5\lambda^{1/2}}\left(\delta+\frac{q^n}{q^n-1}\right)
	\left(\frac{(q^{n}-1)^{1/2}}{q^{n\delta/2}-1}\right)\leq 
	\frac{3}{5\lambda^{1/2}}\left(3+\frac{q^n}{q^n-1}\right)
	\left(\frac{(q^{n}-1)^{1/2}}{q^{3n/2}-1}\right)< 1.$$
	Suppose $\delta=2.$ If $q^n\geq 5$, then 
	$$\frac{3}{5\lambda^{1/2}}\left(\delta+\frac{q^n}{q^n-1}\right)
	\left(\frac{(q^{n}-1)^{1/2}}{q^{n\delta/2}-1}\right)\leq 
	\frac{3}{5\lambda^{1/2}}\left(2+\frac{q^n}{q^n-1}\right)
	\left(\frac{(q^{n}-1)^{1/2}}{q^{n}-1}\right)< 1.$$
	Suppose $\delta=2$ and $q^n=4.$ We have $|H|=3$ so if $\lambda \neq 1$, then
	$$\frac{3\left(\delta+\frac{q^n}{q^n-1}\right)
		\left(\frac{q^{n\delta/2}}{q^{n\delta/2}-1}\right)|H|}{5\lambda^{1/2} |H|^{1/2}q^{n\delta/2}}\leq  \frac{2\cdot 3^{1/2}}{3\cdot \lambda^{1/2}}< 1$$
	Suppose $\delta=2$ and $q^n=3.$ We have $|H|=2$ so if $\lambda>2,$  then
	$$\frac{3\left(\delta+\frac{q^n}{q^n-1}\right)
		\left(\frac{q^{n\delta/2}}{q^{n\delta/2}-1}\right)|H|}{5\lambda^{1/2} |H|^{1/2}q^{n\delta/2}}\leq  \frac{21\cdot 2^{1/2}}{20\cdot \lambda^{1/2}}< 1.$$
	
	\noindent Case 2) $|H|\geq |V|$ (and consequently $n\neq 1$) and $\delta\neq 1.$
	It suffices to prove that
	\begin{equation}\frac{3\left(\delta+\frac{q}{q-1}\right)
		\left(\frac{q^{n\delta/2}}{q^{n\delta/2}-1}\right)q^n}{5|H|^{1/2}q^{n\delta/2}}< 1.
	\end{equation}
	Suppose $q^n\neq 4.$
	$$\frac{3\left(\delta+\frac{q}{q-1}\right)
		\left(\frac{q^{n\delta/2}}{q^{n\delta/2}-1}\right)q^n}{5|H|^{1/2}q^{n\delta/2}}\leq \frac{3\left(2+\frac{q}{q-1}\right)
		\left(\frac{q^{n}}{q^{n}-1}\right)}{5q^{n/2}}< 1.$$
	Suppose $q^n= 4.$ We have $H=\GL(2,2)\cong \perm(3),$ and consequently $|H|=6$ and $p=2/3.$ Hence
	$$\frac{\alpha_U}{|G|^{1/2}}\frac{5}{3}\left(\frac{|U|^{1/2}}{|U|^{1/2}-1}\right)
	\leq \frac{(\delta+2)\cdot \frac 1 p \cdot \frac 3 5 \cdot \frac 4 3}{|H|^{1/2}\cdot 2^\delta}\leq \frac{6}{5\sqrt 6}< 1.$$
	
	\noindent Case 3) $|H| < |V|$ and $\delta=1.$
	Since, $|G|=\lambda|H||V|^\delta$ for some positive integer $\lambda$ it suffices to prove
	\begin{equation}\frac{3\left(\frac{q^n}{q^n-1}\right)\left(\frac{q^{n/2}}{q^{n/2}-1}\right)
		|H|^{1/2}}{5\lambda^{1/2}q^{n/2}}
	< 1.
	\end{equation}
	If $q^n\geq 8,$ or $7 \geq q^n\geq 4$ and $\lambda\neq 1,$  or $q^n=3$ and $\lambda >3,$ then
	$$\frac{3\left(\frac{q^n}{q^n-1}\right)\left(\frac{q^{n/2}}{q^{n/2}-1}\right)
		|H|^{1/2}}{5\cdot \lambda^{1/2}\cdot q^{n/2}}\leq \frac{3\left(\frac{q^n}{q^n-1}\right)\left(\frac{(q^n-1)^{1/2}}{q^{n/2}-1}\right)}{5\cdot \lambda^{1/2}}
	= \frac{3\left(\frac{q^n}{(q^n-1)^{1/2}(q^{n/2}-1)}\right)}{5\cdot \lambda^{1/2}}< 1.$$

	\noindent Case 4) $|H|\geq |V|$ (and consequently $n\neq 1$) and $\delta= 1.$
	It suffices to prove that
	\begin{equation}\label{54}\frac{3\left(\frac{q}{q-1}\right)
		\left(\frac{q^{n/2}}{q^{n/2}-1}\right)}{5|H|^{1/2}q^{n/2}p}< 1.
	\end{equation}
	If $H$ is not a transitive linear group, then $|H|^{1/2}q^{n/2}p\geq 2,$ so
	it suffices to have $$\left(\frac{q}{q-1}\right)
	\left(\frac{q^{n/2}}{q^{n/2}-1}\right) \leq \frac{10}{3},$$
	which is true if $(q,n)\neq (2,2).$ On the other hand, we may exclude the case $(q,n)=(2,2):$ indeed 
	the only soluble irreducible subgroup of $\GL(2,2)$ with order $\geq 4$ is $\GL(2,2),$ which is transitive on the nonzero vectors.
	
	If $H$ is a transitive linear group, then $|H|=(q^n-1)\rho,$ with $\rho$
	the order of the stabilizer in $H$ of a nonzero vector and 
	$$\frac{3\left(\frac{q}{q-1}\right)
		\left(\frac{q^{n/2}}{q^{n/2}-1}\right)}{5|H|^{1/2}q^{n/2}p}\leq \frac{3\left(\frac{q}{q-1}\right)
		\left(\frac{q^{n/2}}{q^{n/2}-1}\right)}{5\sqrt{\rho}},$$
	so it suffices to have
	$$\left(\frac{q}{q-1}\right)
	\left(\frac{q^{n/2}}{q^{n/2}-1}\right) \leq \frac{5\sqrt \rho}{3},$$
	which is true if $q\geq 3$ and if $(q,n,\rho)\notin \{(2,4,2), (2,3,2),$ $(2,3,3), (2,2,2))\}.$  We may exclude the case  $(q,n,\rho)=(2,3,2)$ (there is no transitive linear subgroup of $\GL(3,2)$ of order 14).
	If $(q,n,\rho)=(2,4,2),$ then $H=\GL(1,16)\rtimes C_2$, hence $p=6/30$ so $|H|^{1/2}q^{n/2}p\geq 2$ and (\ref{54}) is true. If $(q,n,\rho)=(2,3,3)$ then
	$H=\gaml(1,8)$ and consequently $p=15/21$ and
	$$\frac{3\left(\frac{q}{q-1}\right)
		\left(\frac{q^{n/2}}{q^{n/2}-1}\right)}{5|H|^{1/2}q^{n/2}p}=\frac{3\cdot 2 \cdot\sqrt 8\cdot 21}{5\cdot (\sqrt 8-1) \cdot 15 \cdot \sqrt{21}\sqrt{8}
	}<1.$$
	If $(q,n,\rho)=(2,2,2)$ then
	$H=\GL(2,2)$ and consequently $p=2/3$ and
	$$\frac{3\left(\frac{q}{q-1}\right)
		\left(\frac{q^{n/2}}{q^{n/2}-1}\right)}{5|H|^{1/2}q^{n/2}p}=\frac{3\cdot 2 \cdot 2\cdot 3}{5\cdot 2 \cdot 2 \cdot \sqrt{6}
	}<1. \qedhere$$
\end{proof}

\begin{lemma}\label{dirab}
	If $G$ is one of the exceptional cases in the statement of Lemma \ref{tecn}, then
	$C(G)< \frac{5}{3}\sqrt{|G|}.$
\end{lemma}
\begin{proof}
	This follows easily by direct computation. We use MAGMA, and the code from \cite[Appendix, page 36]{kz} to compute $C(G)$ explicitly whenever $G$ is a group satisfying the conditions of one of the exceptional cases of Lemma \ref{tecn}.
\end{proof}

The next step is to deal with the case of a central chief factor.
\begin{lemma}\label{13}
	If $G\cong C_p^\delta$, then
	$C(G)\leq \frac{5}{3}\sqrt{|G|},$ with equality if and only if $G=C_2 \times C_2.$
\end{lemma}

\begin{proof}If $p\neq 3$ or $p=2$ and $\delta > 3,$ then
	$$C(G) = \sum_{0\leq i\leq \delta-1}\frac{p^\delta}{p^\delta-p^i}\leq \delta+\frac{p}{(p-1)^2} < \frac{5\cdot p^{\delta/2}}3=\frac{5\cdot \sqrt{|G|}}3.$$
	If $(p,\delta)=(2,1)$ then
	$$\frac{C(G)}{\sqrt{G}}=\frac{2}{\sqrt{2}}=\sqrt{2};$$
	if $(p,\delta)=(2,2)$ then
	$$\frac{C(G)}{\sqrt{G}}=\frac{\frac{4}{2}+\frac{4}{3}}{{2}}=\frac{5}{3};$$
	if $(p,\delta)=(2,3)$ then
	$$\frac{C(G)}{\sqrt{G}}=\frac{\frac{8}{4}+\frac{8}{6}+\frac{8}{7}}{\sqrt{8}}
	\sim 1.5826.\qedhere$$
\end{proof}

\begin{proof}[Proof of Part (ii) of Theorem \ref{KZConjecture}]
	We prove the claim by induction on the order of $|G|.$ If $\frat (G)\neq 1$, then the conclusion follows immediately since $C(G)=C(G/\frat (G)).$ Otherwise $G$ contains a normal subgroup $U$ as in Lemma \ref{sotto}. If $G=U\cong C_p^\delta,$ then the conclusion follows
	from Lemma \ref{13}. Otherwise, Lemma \ref{tecn}, together with the inductive hypothesis gives
	$$C(G)\leq C(G/U)+\alpha_U < \frac{5 \sqrt{|G|}}{3 \sqrt{|U|}}+\frac{ 5(\sqrt{|U|} -1)\sqrt{|G|}}{3\sqrt{|U|}}=\frac{5}{3}\sqrt{|G|}$$
	as claimed.\end{proof}

\end{document}